\documentclass{amsart}
\usepackage{amssymb}
\usepackage[mathcal]{euscript}
\usepackage[cmtip,all]{xy}

\newcommand{\bydef}{:=}

\newcommand{\id}{\mathrm{id}}

\newcommand{\cA}{\mathcal{A}}

\newcommand{\frg}{{\mathfrak g}}

\newcommand{\ZZ}{\mathbb{Z}}

\newcommand{\QQ}{\mathbb{Q}}

\newcommand{\CC}{\mathbb{C}}

\newcommand{\FF}{\mathbb{F}}
\newcommand{\KK}{\mathbb{K}}

\DeclareMathOperator{\Grp}{\mathrm{Grp}}
\DeclareMathOperator{\AlgF}{\mathrm{Alg_{\FF}}}
\DeclareMathOperator{\AlgK}{\mathrm{Alg_{\KK}}}
\DeclareMathOperator{\Hom}{\mathrm{Hom}}

\DeclareMathOperator{\Aut}{\mathrm{Aut}}
\DeclareMathOperator{\AAut}{\mathbf{Aut}}

\DeclareMathOperator{\supp}{\mathrm{Supp}\,}

\newcommand{\frsl}{{\mathfrak{sl}}}

\newcommand{\fru}{{\mathfrak{u}}}

\newcommand{\Gs}{\mathbf{G}}
\newcommand{\Hs}{\mathbf{H}}

\newtheorem{theorem}{Theorem}
\newtheorem{proposition}[theorem]{Proposition}

\theoremstyle{definition}

\theoremstyle{remark}
\newtheorem{remark}[theorem]{Remark}

\begin{document}

\title[Gradings on algebras over algebraically closed fields]{Gradings on algebras\\ over algebraically closed fields}

\author[Alberto Elduque]{Alberto Elduque${}^\star$}
\address{Departamento de Matem\'{a}ticas
 e Instituto Universitario de Matem\'aticas y Aplicaciones,
 Universidad de Zaragoza, 50009 Zaragoza, Spain}
\email{elduque@unizar.es}
\thanks{${}^\star$Supported by the Spanish Ministerio de Econom\'{\i}a y Competitividad---Fondo Europeo de Desarrollo Regional (FEDER) MTM2010-18370-C04-02 and by the Diputaci\'on General de Arag\'on---Fondo Social Europeo (Grupo de Investigaci\'on de \'Algebra)}



\date{}

\begin{abstract}
The classification, both up to isomorphism or up to equivalence, of the gradings on a finite dimensional nonassociative algebra $\cA$ over an algebraically closed field $\FF$ such that the group scheme of automorphisms $\AAut(\cA)$ is smooth is shown to be equivalent to the corresponding problem for $\cA_\KK=\cA\otimes_\FF\KK$ for any algebraically closed field extension $\KK$.
\end{abstract}

\maketitle


\section{Introduction}

Gradings on Lie algebras are ubiquitous, as shown in the introduction of \cite{LGI}. This paper by Patera and Zassenhaus started a systematic research on gradings by abelian groups on simple finite dimensional Lie algebras over algebraically closed fields of characteristic $0$. In the sequel \cite{LGII}, a description was given of the fine gradings on the simple classical Lie algebras, other than $D_4$, over the field of complex numbers $\CC$. The classification of the fine gradings on the classical simple Lie algebras, including $D_4$, over any algebraically closed field of characteristic $0$ was achieved in \cite{EldFine}. A complete account of these results and of the state of the art appears in \cite{EK}.

For the exceptional simple Lie algebras, the gradings on $G_2$ and $F_4$ are related (see \cite{EK} and references there in) to gradings on the octonions and on the Albert algebra (the exceptional simple Jordan algebra). For $E_6$, the classification of fine gradings is achieved in \cite{DV}. At certain point, Draper and Viruel make use of known results over the complex numbers, where the corresponding simple Lie groups have been thoroughly studied, and transfer these results to arbitrary algebraically closed fields of characteristic $0$ in an ad hoc way (see \cite[Proposition 2]{DV}).

The goal of this paper is to use an importat result by Margaux \cite[Theorem 1.1]{M} on algebraic group schemes over algebraically closed fields to show that, given a finite dimensional nonassociative (i.e., not necessarily associative) algebra $\cA$ over an algebraically closed field $\FF$ such that the (affine) group scheme of automorphisms $\AAut(\cA)$ is smooth (this is automatically satisfied if the characteristic is $0$), and given a field extension $\KK/\FF$ for an algebraically closed field $\KK$, the classifications of gradings up to isomorphism or up to equivalence for $\cA$ and for $\cA_\KK\bydef \cA\otimes_\FF\KK$ are equivalent: Any grading on $\cA_\KK$ is isomorphic or equivalent to a  grading obtained by extension of a unique (up to isomorphism or equivalence) grading on $\cA$.

In particular, given a finite dimensional simple Lie algebra $\frg$ over an algebraically closed field $\FF$ of characteristic $0$, $\frg$ is defined over the algebraic closure $\overline{\QQ}$ of the rational numbers: $\frg=\frg_0\otimes_{\overline{\QQ}}\FF$ for a (unique) simple Lie algebra $\frg_0$ over $\overline{\QQ}$. Hence the classifications of gradings on $\frg$ and on $\frg_0$ are equivalent, and so are the classifications of gradings on $\frg_0$ and on $\frg_\CC\bydef \frg_0\otimes_{\overline{\QQ}}\CC$.

The conclusion is that
\begin{center}
\emph{it is enough to classify gradings over the complex numbers!}
\end{center}

\medskip

In the next section, the necessary definitions and results on gradings by abelian groups will be reviewed in a way suitable for our purposes. The result of Margaux will be recalled in Section 3. As a consequence, the equivalence of the classifications of gradings up to isomorphism mentioned above will be derived quickly. Some extra arguments are needed for the equivalence of the classifications up to equivalence.

\bigskip

\section{Gradings}

This section will review some basic facts on gradings that will be needed later on. The reader may consult the first chapter of \cite{EK} for details.

\smallskip

Let $\cA$ be a nonassociative algebra over a field $\FF$ and let $G$ be an abelian group. A \emph{grading} on $\cA$ by $G$, or \emph{$G$-grading}, is a vector space decomposition
\[
\Gamma: \cA=\bigoplus_{g\in G}\cA_g,
\]
satisfying $\cA_g\cA_h\subseteq \cA_{gh}$ for all $g,h\in G$. In this case, the nonzero elements in $\cA_g$ are \emph{homogeneous of degree} $g$ and we write $\deg x=g$. The subspace $\cA_g$ is the \emph{homogeneous component} of degree $g$.

The set
\[
\supp\Gamma\bydef \{g\in G : \cA_g\ne 0\}
\]
is called the \emph{support} of $\Gamma$. Without loss of generality we may, and will, restrict to the case where $G$ is generated by the support.

\smallskip

Gradings on $\cA$ may be compared in two ways, depending on whether the group $G$ is taken as part of the definition.

Two $G$-gradings $\Gamma:\cA=\bigoplus_{g\in G}\cA_g$ and $\Gamma':\cA=\bigoplus_{g\in G}\cA_g'$ are \emph{isomorphic} if there is an automorphism $\varphi\in \Aut(\cA)$ such that $\varphi(\cA_g)=\cA_g'$ for any $g\in G$. We say then that $\varphi:\Gamma\rightarrow \Gamma'$ is an isomorphism.

On the other hand, given two abelian groups $G$ and $H$ and gradings $\Gamma:\cA=\bigoplus_{g\in G}\cA_g$ and $\Gamma':\cA=\bigoplus_{h\in H}\cA_h'$, then $\Gamma$ and $\Gamma'$ are said to be \emph{equivalent} if there is a bijection $\alpha:\supp\Gamma\rightarrow\supp\Gamma'$ and an automorphism $\varphi\in\Aut(\cA)$ such that $\varphi(\cA_g)=\cA_{\alpha(g)}'$ for any  $g\in\supp\Gamma$.

\smallskip

For any group homomorphism $\rho:G\rightarrow H$ and any $G$-grading $\Gamma:\cA=\bigoplus_{g\in G}\cA_g$, the decomposition
\[
{}^\rho\Gamma:\cA=\bigoplus_{h\in H}\cA_h',
\]
with $\cA_h'\bydef \sum_{\rho(g)=h}\cA_g$, is the \emph{grading induced from $\Gamma$ by $\rho$}. The new grading ${}^\rho\Gamma$ is an example of a \emph{coarsening}.

A grading $\Gamma:\cA=\bigoplus_{g\in G}\cA_g$ is a \emph{refinement} of another grading $\Gamma':\cA=\bigoplus_{h\in H}\cA_h'$ if for any $g\in\supp\Gamma$, there is an element $h\in\supp\Gamma'$ such that $\cA_g\subseteq \cA_h'$. In other words, for any $h\in\supp\Gamma'$, the homogeneous component $\cA_h'$ is the (direct) sum of some of the homogeneous components of $\Gamma$. The grading $\Gamma'$ is then said to be a \emph{coarsening} of $\Gamma$. The refinement is proper if for some $g\in\supp\Gamma$, the containment $\cA_g\subseteq \cA_h'$ above is strict. The grading $\Gamma$ is called \emph{fine} if it does not admit any proper refinement. The root space decomposition of any finite dimensional semisimple Lie algebra over an algebraically closed field of characteristic $0$ is an example of fine grading.

\smallskip

A grading on the algebra $\cA$ may be realized by different groups. For instance, consider the Lie algebra $\frsl_2(\FF)$ of $2\times 2$ matrices of trace zero and its basis $\{H,E,F\}$ with
\[
H=\left(\begin{smallmatrix} 1&0\\ 0&-1\end{smallmatrix}\right),\quad
E=\left(\begin{smallmatrix} 0&1\\ 0&0\end{smallmatrix}\right),\quad
F=\left(\begin{smallmatrix} 0&0\\ 1&0\end{smallmatrix}\right).
\]
This Lie algebra is graded by the integers modulo $n$, for any $n\geq 3$, with the same homogeneous components: $\deg E=1$, $\deg H=0$, $\deg F=-1$. So this `same' grading can be realized as a grading by $\ZZ/n\ZZ$ for any $n\geq 3$. However, the `natural grading group' in this situation is $\ZZ$.

More precisely, given a grading $\Gamma:\cA=\bigoplus_{g\in G}\cA_g$, let $U(\Gamma)$ be the group with generators $\{\alpha_g: g\in\supp\Gamma\}$ and relations $\alpha_g\alpha_h=\alpha_{gh}$ in case $g,h,gh\in\supp\Gamma$. The grading $\Gamma$ is realized as a grading $\widehat{\Gamma}$ by $U(\Gamma)$ with support $\{\alpha_g: g\in \supp\Gamma\}$ (the set of generators of $U(\Gamma)$), where the homogeneous component of degree $\alpha_g$ is precisely $\cA_g$. Moreover, $U(\Gamma)$ is characterized by the following universal property: given any realization of $\Gamma$ as a grading by an abelian group $H$, that is, given an $H$-grading $\Gamma':\cA=\bigoplus_{h\in H}\cA_h'$ such that for any $h\in \supp\Gamma'$ there is a $g\in \supp\Gamma$ with $\cA_h'=\cA_g$, there is a unique homomorphism $\rho:U(\Gamma)\rightarrow H$ such that $\Gamma'={}^\rho\widehat{\Gamma}$.

\begin{remark}\label{re:alpha_Gamma} Two gradings $\Gamma$ and $\Gamma'$ are then equivalent if and only if there is an automorphism $\varphi\in\Aut(\cA)$ and a group isomorphism $\alpha:U(\Gamma)\rightarrow U(\Gamma')$ such that
$\varphi$ is an isomorphism ${}^\alpha\Gamma\rightarrow \Gamma'$ of $U(\Gamma')$-gradings.
\end{remark}

\smallskip

Assume from now on that the algebra $\cA$ is finite dimensional. Then the group scheme of automorphisms $\AAut(\cA)$ is representable (i.e., affine). Recall that $\AAut(\cA)$ is the functor
\[
\begin{split}
\AAut(\cA): \AlgF&\longrightarrow \Grp\\
R\ &\mapsto\ \Aut_R(\cA\otimes_\FF R)\quad \text{(automorphisms as algebras over $R$)},
\end{split}
\]
with the natural action on homomorphisms, where $\AlgF$ denotes the category of unital commutative and associative algebras over $\FF$ and $\Grp$ the category of groups.

Any $G$-grading $\Gamma:\cA=\bigoplus_{g\in G}\cA_g$ induces a homomorphism of affine group schemes (see, for instance, \cite[Chapter 1]{EK})
\[
\theta_\Gamma: G^D\rightarrow \AAut(\cA),
\]
where $G^D$ is the dual of the constant group scheme determined by $G$. That is, $G^D(R)=\Hom_{\AlgF}(\FF G,R)$ for any object $R$ in $\AlgF$. Here $\FF G$ denotes the group algebra of $G$. The homomorphism $\theta_\Gamma$ is defined as follows. For any algebra homomorphism $\chi:\FF G\rightarrow R$ (i.e., $\chi\in G^D(R)$), $\theta_\Gamma(\chi)$ is the automorphism
\[
\begin{split}
\theta_\Gamma(\chi):\cA\otimes_\FF R&\longrightarrow \cA\otimes_\FF R,\\
a\otimes r\,&\mapsto\ a\otimes \chi(g)r,
\end{split}
\]
for any $g\in G$, $a\in\cA_g$ and $r\in R$.

Conversely, given any homomorphism $\theta:G^D\rightarrow \AAut(\cA)$ and the `generic element' in $G^D$: $\id_{\FF G}\in G^D(\FF G)$, the image $\theta(\id_{\FF G}):\cA\otimes_\FF \FF G\rightarrow \cA\otimes_\FF \FF G$ is an automorphism of algebras over $\FF G$. Consider, for any $g\in G$, the subspace
\[
\cA_g\bydef\{ a\in\cA : \theta(\id_{\FF G})(a\otimes 1)=a\otimes g\}.
\]
Then $\cA$ is the direct sum $\bigoplus_{g\in G}\cA_g$ and this defines a $G$-grading $\Gamma_\theta$ of $\cA$. The correspondences $\Gamma\mapsto \theta_\Gamma$ and $\theta\mapsto \Gamma_\theta$ are inverses of each other.

We thus conclude that the set of $G$-gradings on $\cA$ can be identified with the set of homomorphisms $\Hom(G^D,\AAut(\cA))$.

Besides, if two $G$-gradings $\Gamma:\cA=\bigoplus_{g\in G}\cA_g$ and $\Gamma':\cA=\bigoplus_{g\in G}\cA'_g$ are isomorphic and $\varphi\in\Aut(\cA)$\,($=\AAut(\cA)(\FF)$, the group of rational points of $\AAut(\cA)$) satisfies $\varphi(\cA_g)=\cA'_g$ for any $g\in G$, the conjugation by $\varphi$ gives an automorphism $\Omega_\varphi$ of $\AAut(\cA)$ (for any $R$ we consider the conjugation by $\varphi\otimes \id:\cA\otimes_\FF R\rightarrow \cA\otimes_\FF R$) such that $\theta_{\Gamma'}=\Omega_\varphi\circ \theta_\Gamma$, and conversely.

Therefore, the isomorphism classes of $G$-gradings on $\cA$ are in bijection with
\[
\Hom(G^D,\AAut(\cA))/\Aut(\cA),
\]
the set of conjugacy classes of homomorphisms $G^D\rightarrow \AAut(\cA)$ under the natural action of $\Aut(\cA)=\AAut(\cA)(\FF)$ by conjugation (see \cite[Proposition 1.36]{EK}).

\begin{remark}\label{re:coarsening}
Given a grading $\Gamma:\cA=\bigoplus_{g\in G}\cA_g$ on the algebra $\cA$ by the abelian group $G$, and a group homomorphism $\alpha:G\rightarrow H$ for another abelian group $H$, $\alpha$ induces naturally a homomorphism $\alpha^D:H^D\rightarrow G^D$ and
\[
\theta_{{}^\alpha\Gamma}=\theta_\Gamma\circ\alpha^D.
\]
\end{remark}

Given a field extension $\KK/\FF$, the group scheme $\AAut(\cA_\KK)$ of $\cA_\KK\bydef \cA\otimes_\FF \KK$ coincides with the `restriction' $\AAut(\cA)_\KK$ of $\AAut(\cA)$ to the subcategory $\AlgK$.

In this situation we will write $G_\FF^D$ to denote the affine group scheme over $\FF$ represented by the group algebra $\FF G$, and by $G_\KK^D$ the affine group scheme over $\KK$ represented by the group algebra $\KK G$. Again, $G_\KK^D$ is the restriction $\bigl(G_\FF^D\bigr)_\KK$.

Also, given a $G$-grading $\Gamma:\cA=\bigoplus_{g\in G}\cA_g$, $\Gamma$ induces naturally a grading $\Gamma_\KK$ on $\cA_\KK$, where $(\cA_\KK)_g\bydef \cA_g\otimes_\FF \KK$ for any $g\in \supp\Gamma$. Then the homomorphism $\theta_{\Gamma_\KK}: G_\KK^D\rightarrow \AAut(\cA_\KK)$ is the restriction of $\theta_\Gamma$.

\bigskip

\section{Main results}

In 2009, Margaux proved the following general result:

\begin{theorem}[{\cite[Theorem 1.1]{M}}]\label{th:Margaux}
Let $\Gs$ and $\Hs$ be algebraic group schemes over an algebraically closed field $\FF$ and let $\KK/\FF$ be an algebraically closed field extension. If $\Gs$ is affine and linearly reductive and $\Hs$ is smooth, then every $\KK$-homomorphism $\Gs_\KK\rightarrow \Hs_\KK$ is $\Hs(\KK)$-conjugate to one induced by a $\FF$-homomorphism $\Gs\rightarrow \Hs$. More precisely, the natural map
\[
\Hom_\FF(\Gs,\Hs)/\Hs(\FF)\longrightarrow \Hom_\KK(\Gs_\KK,\Hs_\KK)/\Hs(\KK)
\]
is bijective.
\end{theorem}

The affine group schemes of the form $G^D$ for finitely generated groups $G$, are the diagonalizable group schemes, and hence they are linearly reductive. When dealing with gradings on a finite dimensional algebra, the grading group $G$ may always be taken to be finitely generated (by the finite support), and hence the affine scheme $G^D$ is algebraic. Also, for any finite dimensional nonassociative algebra $\cA$, the scheme 
$\AAut(\cA)$ is algebraic. Recall too that over fields of characteristic zero, any affine group scheme is smooth.

\smallskip

Therefore, as a direct consequence of Theorem \ref{th:Margaux} we obtain the following result.

\begin{theorem}\label{th:isom_classes}
Let $\cA$ be a finite dimensional nonassociative algebra over an algebraically closed field $\FF$ satisfying that $\AAut(\cA)$ is smooth, let $\KK/\FF$ be an algebraically closed field extension, and let $G$ be a finitely generated group. For any $G$-grading $\Gamma$ on $\cA$ denote by $[\Gamma]$ the corresponding isomorphism class. Then the map
\[
\begin{split}
\left\{\begin{matrix}
\text{isomorphism classes of}\\
\text{$G$-gradings on $\cA$}
\end{matrix}\right\}
&\longrightarrow
\left\{\begin{matrix}
\text{isomorphism classes of}\\
\text{$G$-gradings on $\cA_\KK$}
\end{matrix}\right\}\\[4pt]
[\Gamma]\qquad&\mapsto\qquad [\Gamma_\KK]
\end{split}
\]
is a bijection.
\end{theorem}
\begin{proof}
The set of isomorphism classes of $G$-gradings on $\cA$ is in bijection with $\Hom_\FF(G_\FF^D,\AAut(\cA))/\Aut(\cA)$, with the class $[\Gamma]$ of a $G$-grading corresponding to the conjugacy class of the homomorphism $\theta_\Gamma: G_\FF^D\rightarrow \AAut(\cA)$. Therefore the result follows immediately from Theorem \ref{th:Margaux}.
\end{proof}

The next result deals with equivalence classes of gradings:

\begin{proposition}
Let $\cA$ be a finite dimensional nonassociative algebra over an algebraically closed field $\FF$ satisfying that $\AAut(\cA)$ is smooth and let $\KK/\FF$ be an algebraically closed field extension. Let $G$ and $H$ be two abelian groups and let $\Gamma:\cA=\bigoplus_{g\in G}\cA_g$ be a $G$-grading on $\cA$ and $\Gamma':\cA=\bigoplus_{h\in H}\cA'_g$ an $H$-grading on $\cA$. Then $\Gamma$ is equivalent to $\Gamma'$ if and only if $\Gamma_\KK$ is equivalent to $\Gamma'_\KK$.
\end{proposition}
\begin{proof}
If $\Gamma$ is equivalent to $\Gamma'$, there is a bijection $\alpha:\supp\Gamma\rightarrow \supp \Gamma'$ and an automorphism $\varphi\in\Aut(\cA)$ such that $\varphi(\cA_g)=\cA'_{\alpha(g)}$ for any $g\in \supp\Gamma$. But $\supp(\Gamma_\KK)=\supp(\Gamma)$, $\supp(\Gamma'_\KK)=\supp(\Gamma')$ and $\varphi$ extends to an automorphism of $\cA_\KK$. Thus, $\Gamma_\KK$ is equivalent to $\Gamma'_\KK$.

Conversely, assume that $\Gamma_\KK$ and $\Gamma'_\KK$ are equivalent. Denote by $K$ and $K'$ the universal groups $U(\Gamma)=U(\Gamma_\KK)$ and $U(\Gamma')=U(\Gamma'_\KK)$ respectively, and consider both $\Gamma$ and $\Gamma'$ realized as gradings by their universal groups. Then there exists a group isomorphism $\alpha:K\rightarrow K'$ such that ${}^\alpha(\Gamma_\KK)$ is isomorphic to $\Gamma'_\KK$ (see Remark \ref{re:alpha_Gamma}). But ${}^\alpha(\Gamma_\KK)$ equals $({}^\alpha\Gamma)_\KK$, so Theorem \ref{th:isom_classes} shows that ${}^\alpha\Gamma$ is isomorphic to $\Gamma'$, and this proves that $\Gamma$ and $\Gamma'$ are equivalent.
\end{proof}

\begin{proposition}
Let $\cA$ be a finite dimensional nonassociative algebra over an algebraically closed field $\FF$ satisfying that $\AAut(\cA)$ is smooth and let $\KK/\FF$ be an algebraically closed field extension. Let $G$ be an abelian group and let $\Gamma:\cA=\bigoplus_{g\in G}\cA_g$ be a $G$-grading on $\cA$. Then $\Gamma$ is fine if and only if so is $\Gamma_\KK$.
\end{proposition}
\begin{proof}
Any proper refinement of $\Gamma$ induces a proper refinement of $\Gamma_\KK$. Hence, if $\Gamma_\KK$ is fine, so is $\Gamma$.

Assume now that $\Gamma_\KK$ is not fine and let $\tilde\Gamma$ be a proper refinement. Take $H=U(\tilde\Gamma)$. Then there is a group homomorphism
\[
\alpha:H=U(\tilde\Gamma)\rightarrow G
\]
such that $\Gamma_\KK={}^\alpha\tilde\Gamma$.

The refinement is proper, so there exists an element $\hat g\in\supp(\Gamma)=\supp(\Gamma_\KK)$ such that $\alpha^{-1}(\hat g)\cap\supp(\tilde\Gamma)$ contains at least two elements. Hence we have $\tilde\Gamma:\cA_\KK=\bigoplus_{h\in H}(\tilde\cA_\KK)_h$ and for any $g\in G$
\[
(\cA_g)_\KK=(\cA_\KK)_g=\bigoplus_{h\in\alpha^{-1}(g)}(\tilde\cA_\KK)_h.
\]
Moreover,
\begin{equation}\label{eq:1}
\theta_{\Gamma_\KK}=\theta_{\tilde\Gamma}\circ\alpha_\KK^D
\end{equation}
(see Remark \ref{re:coarsening}), where $\alpha_\KK^D:G_\KK^D\rightarrow H_\KK^D$ is the natural homomorphism induced by $\alpha$.

By Theorem \ref{th:isom_classes} there exists an $H$-grading $\widehat{\Gamma}:\cA=\bigoplus_{h\in H}\hat\cA_h$ on $\cA$ such that $[\tilde\Gamma]=[\widehat{\Gamma}_\KK]$. In particular we get $\supp(\tilde\Gamma)=\supp(\widehat{\Gamma}_\KK)=\supp(\widehat{\Gamma})$.

Hence there exists an automorphism $\varphi\in\Aut(\cA_\KK)$ such that
\begin{equation}\label{eq:2}
\theta_{\tilde\Gamma}=\Omega_\varphi\circ\theta_{\widehat{\Gamma}_\KK}.
\end{equation}
(Recall that $\Omega_\varphi$ denotes conjugation by $\varphi$.)

Equations \eqref{eq:1} and \eqref{eq:2} give:
\[
\begin{split}
\theta_{\Gamma_\KK}=\theta_{\tilde\Gamma}\circ\alpha_\KK^D&=
\Omega_\varphi\circ\theta_{\widehat{\Gamma}_\KK}\circ\alpha_\KK^D\\
    &=\Omega_\varphi\circ\bigl(\theta_{\widehat{\Gamma}}\circ\alpha^D\bigr)_\KK\\
    &=\Omega_\varphi\circ\bigl(\theta_{{}^\alpha\widehat{\Gamma}}\bigr)_\KK.
\end{split}
\]
Therefore, the $G$-gradings $\Gamma_\KK$ and $\bigl({}^\alpha\widehat{\Gamma}\bigr)_\KK$ are isomorphic and hence, by Theorem \ref{th:isom_classes}, so are the $G$-gradings $\Gamma$ and ${}^\alpha\widehat{\Gamma}$. But ${}^\alpha\widehat{\Gamma}$ is a proper coarsening of $\widehat{\Gamma}$ because $\alpha^{-1}(\hat g)\cap\supp(\widehat{\Gamma})=\alpha^{-1}(\hat g)\cap\supp(\tilde\Gamma)$ contains at least two elements. We conclude that ${}^\alpha\widehat{\Gamma}$ is not fine, and neither is $\Gamma$.
\end{proof}

Our last result is a straightforward consequence of the results above:

\begin{theorem}\label{th:equiv_classes}
Let $\cA$ be a finite dimensional nonassociative algebra over an algebraically closed field $\FF$ satisfying that $\AAut(\cA)$ is smooth and let $\KK/\FF$ be an algebraically closed field extension. For a grading $\Gamma:\cA=\bigoplus_{g\in G}\cA_g$, denote by $[\Gamma]_{\text{eq}}$ its equivalence class. Then the map
\[
\begin{split}
\left\{\begin{matrix}
\text{equivalence classes of}\\
\text{gradings on $\cA$}
\end{matrix}\right\}
&\longrightarrow
\left\{\begin{matrix}
\text{equivalence classes of}\\
\text{gradings on $\cA_\KK$}
\end{matrix}\right\}\\[4pt]
[\Gamma]_{\text{eq}}\qquad&\mapsto\qquad [\Gamma_\KK]_{\text{eq}}
\end{split}
\]
is a bijection, which restricts to a bijection
\[
\begin{split}
\left\{\begin{matrix}
\text{equivalence classes of}\\
\text{fine gradings on $\cA$}
\end{matrix}\right\}
&\longrightarrow
\left\{\begin{matrix}
\text{equivalence classes of}\\
\text{fine gradings on $\cA_\KK$}
\end{matrix}\right\}\\[4pt]
[\Gamma]_{\text{eq}}\qquad&\mapsto\qquad [\Gamma_\KK]_{\text{eq}}.
\end{split}
\]
\end{theorem}

\smallskip

\begin{remark}
For any two real compact Lie groups $A,B$, the natural map
\[
\Hom(A,B)/B\longrightarrow \Hom(\underline{A}(\CC),\underline{B}(\CC))/\underline{B}(\CC)
\]
is a bijection \cite{AYY}, where $\underline{A}(\CC)$ (respectively $\underline{B}(\CC)$) is the complexification of $A$ (resp. $B$). On the left we consider the category of compact Lie groups, while on the right the category of complex reductive (algebraic) groups.

As a particular case, given a compact simple Lie algebra $\fru$, the conjugacy classes of abelian subgroups in $\Aut(\fru)$ are in bijection with the conjugacy classes of quasitori in $\Aut(\fru_\CC)$. This shows that the classification of `weak isomorphism classes of gradings' in $\fru_\CC$ with the property that the support generates the grading group is equivalent to the classification of the conjugacy classes of abelian subgroups in $\Aut(\fru)$ (see \cite[Proposition 1.32]{EK}), and in particular the classification of fine gradings in $\fru_\CC$, up to equivalence, is equivalent to the classification of the conjugacy classes of maximal abelian subgroups in $\Aut(\fru)$ (\cite[Proposition 1.32]{EK}).

With this in mind, the recent results by Jun Yu \cite{Yu1,Yu2,Yu3} classifying the conjugacy classes of a specific set of abelian subgroups of compact simple Lie groups, which contains the maximal abelian subgroups, give the classification, up to equivalence, of the fine gradings on the finite dimensional simple Lie algebras over $\CC$ and hence, by Theorem \ref{th:equiv_classes}, over any algebraically closed field of characteristic $0$. These results give an affirmative answer to \cite[Question 6.65]{EK}: the list of fine gradings on the exceptional simple Lie algebras in \cite[\S 6.6]{EK} is complete.
\end{remark}


\end{document}